\documentclass[reqno,onecolumn,oneside]{paper}
\usepackage[english]{babel}
\usepackage{enumerate,amsmath,amsthm,amsfonts,
amssymb,latexsym,mathrsfs,
}
\usepackage[normalem]{ulem}
\usepackage[all]{xy}
\usepackage{tikz-cd}
\usepackage[sort]{cite}

\usepackage[mathscr]{euscript}
 \let\mathscr\relax
\usepackage[scr]{rsfso}
\usepackage{xcolor}
\usepackage{cancel}

\newtheorem{theorem}{Theorem}[section]
\newtheorem{proposition}[theorem]{Proposition}
\newtheorem{lemma}[theorem]{Lemma}
\newtheorem{corollary}[theorem]{Corollary}
\theoremstyle{definition}
\newtheorem{definition}[theorem]{Definition}

\theoremstyle{remark}

\newcommand{\End}{\operatorname{End}}

\newcommand{\id}{\operatorname{id}}

\newcommand{\dps}{\displaystyle}

\newcommand{\SL}{\mathcal{SL}}

\newcommand{\cat}{\mathcal}

\renewcommand{\phi}{\varphi}

\newcommand{\la}{\left\langle}
\newcommand{\ra}{\right\rangle}

\newcommand{\under}{\backslash}

\newcommand{\To}{\Rightarrow}

\newcommand{\ov}{\overline}

\newcommand{\tensor}{\otimes}

\newcommand{\mc}[1][Q]{{}_{#1}{\mathcal{M}}}
\newcommand{\MC}[1][Q]{{\mathcal{M}}_{#1}}
\newcommand{\tr}{\operatorname{tr}}

\usepackage{xparse}
\NewDocumentCommand{\mm}{ O{M} O{Q} m }{{}_{#2}{#1}_{#3}}
 
\newcommand{\QM}{\mm{}}

\newcommand{\QMR}{\mm{R}}

\def\amslatex\slash{{\protect\AmS-\protect\LaTeX}}

\begin{document} 

\title{Morita Equivalence for Quantales}

\author{\renewcommand{\thefootnote}{\arabic{footnote}}
\rm Moacyr Rodrigues\footnotemark[1] \and
\renewcommand{\thefootnote}{\arabic{footnote}}
\rm Ciro Russo\footnotemark[2]}

\maketitle
\footnotetext[1]{Departamento de Matem\'atica, Universidade Federal da Bahia -- Salvador, Bahia, Brazil.\\
Departamento de Ciências Exatas, Universidade Estadual de Feira de Santana -- Feira de Santana, Bahia, Brazil. 
\\ {\tt mrmjunior@uefs.br}}
\footnotetext[2]{Departamento de Matem\'atica, Universidade Federal da Bahia -- Salvador, Bahia, Brazil. \\ {\tt ciro.russo@ufba.br}}

\maketitle
\date{}

\begin{abstract}
We study Morita equivalence in the context of quantales with identity, in the wake of Katsov and Nam's analogous work on semirings. Among a number of other results, we prove a characterization of Morita equivalence and an Eilenberg-Watts-type Theorem for quantales.
\end{abstract}

\section{Introduction}

Quantales were introduced by Mulvey \cite{mul} in 1986. They boast nowadays a quite extensive literature and they proved to be of interest for various areas of mathematics, such as, for example, logic \cite{abrvic,yetter,galtsi,rusapal,ruslu}, non-commutative topology \cite{borvdb,conmir}, and fuzzy topology \cite{zhang2022}. There is a certain similarity between quantales and rings (or, even better, semirings) under various aspects, and indeed many classical properties and constructions of rings and semirings have a quantale-theoretic version. Moreover, the study of quantales is very often connected to that of quantale modules, the latter constituting essentially the representation theory of quantales.

What is nowadays called Morita equivalence was introduced in 1958 by Kiiti Morita \cite{morita} for rings, and eventually extended in various directions (see, for example, \cite{moreq2,moreq3,moreq4,katnam}). The literature on quantale modules is relatively recent, but already quite well-established \cite{mesa2025,rusapal,russajl,rusjlc,pas,sol,rose94,nief96,pas99}, and various approaches to Morita equivalence for quantales have been proposed at least since 2001 \cite{pas2001,pas2002, pas2005a, pas2005b, luh, mesa2025}. 

The aim of this work is to study Morita equivalence in the general context of quantales with identity, with no further assumptions, in the wake of Katsov and Nam's work on semirings \cite{katnam}. As we shall see, most of the results that hold for semirings can be proved for quantales, possibly up to a reformulation, including a suitable version of the Eilenberg-Watts Theorem \cite{eil60,wat60}. However, we chose to follow a more classical approach, i.e., the one for which the definition of Morita equivalence is based on the equivalence of module categories, while the adjunction involving the tensor product and the hom-functor appears in a characterization of the equivalence. Besides that, the main differences between our results and the ones presented in \cite{katnam} have to do with finiteness conditions. For example, Morita equivalence between two quantales can be characterized by the existence of a progenerator acting as a ``bridging object'' between the two categories of modules, but this is not equivalent to the existence of such a bridging object which is also finitely generated. Other differences of this kind can be seen inside the proofs whenever direct sums are involved in the case of semirings. Indeed, direct sums of semimodules over semirings are the resulting objects of categorical coproducts, while, in the categories of quantale modules, product and coproduct of a family of modules yield the same object, and only differ in the direction of the arrows. Last, it is worth mentioning that several homomogical applications to subtractive and additively regular semirings are not possible for quantales, because such properties never hold in non-trivial quantales.

In \cite{pas2001,pas2002, pas2005a,pas2005b}, the author presented several results on Morita theory for quantales, focusing mainly on m-regular quantales, while \cite{luh} approached the topic within the context of quantale enlargements. The more recent paper \cite{mesa2025} focuses on the characterization of Morita equivalent quantales starting from a purely categorical approach. In that paper, the existence of which we discovered after the completion of the first draft of the present one, the author presented various results that we also show here, but using mainly properties of enriched categories. In that context, the proofs of such results are often shorter and elegant but, on the other hand, many details on how certain isomorphisms arise and look like are not explicit. In this paper, we approached Morita equivalence in a more classical algebraic fashion, therefore some results may look somehow lengthy than the corresponding ones obtain via enriched category theory, but they have the advantage of explicitly describing how things concretely work.

The paper is organized as follows. In Section \ref{prel}, we shall briefly recall definitions and results on quantales and their modules.

Section \ref{tensor} is dedicated to tensor products of quantale modules. Besides recalling the definition and the existence theorem, we prove various results that will be useful in the subsequent sections.

Projective generators of the categories of quantale modules are introduced and characterized in Section \ref{progenerator}. Such modules are directly involved in the description of Morita equivalent quantales.

Finally, in Section \ref{morita}, we shall investigate the properties of Morita equivalence of quantales and prove a quantale-theoretic version of Eilenberg-Watts Theorem (Theorem \ref{adj}) and a characterization of Morita equivalence by means of progenerators (Theorem \ref{moritath}).

\section{Preliminaries}
\label{prel}

In this section we shall briefly recall definitions and results on the ordered algebraic structures directly involved in our main results. For any further information these topics, we refer the reader to \cite{krupas,mul,rosenthal} concerning quantales, and to \cite{pas,mulnaw,rusthesis,rusjlc,rusapal,russajl,sol} regarding quantale modules.


The category $\SL$ of \emph{sup-lattices} has complete lattices as objects and maps preserving arbitrary joins -- or, equivalently, if the orders are complete, residuated maps -- as morphisms. The bottom element of a sup-lattice shall be denoted by $\bot$ and the top element by $\top$. We recall that any sup-lattice morphism obviously preserves bottom, while it does not need to preserve top. 

Quantales are often defined as sup-lattices in the category of semigroups. Since we shall only deal with unital quantales, we can say that $(Q, \bigvee, \cdot, 1)$ is a \emph{quantale} if $(Q,\bigvee)$ is a sup-lattice, $(Q, \cdot, 1)$ is a monoid, and the product is biresiduated, i.~e., for all $a,b \in Q$,
$$\exists b\under a = \max\{c \in Q \mid bc \leq a\} \text{ and } \exists a/b = \max\{c \in Q \mid cb \leq a\}.$$
The above condition is equivalent to the distributivity of $\cdot$ w.r.t. any join:
$$\forall a \in Q \ \forall B \subseteq Q \ \left(a \cdot \bigvee B = \bigvee\limits_{b \in B} (a \cdot b) \text{ and } \left(\bigvee B\right) \cdot a = \bigvee\limits_{b \in B} (b \cdot a)\right).$$
A quantale is \emph{commutative} if its multiplication is commutative and \emph{integral} if $1 = \top$.

The morphisms in the category $\cat Q$ of quantales are maps that are simultaneously sup-lattice and monoid homomorphisms or, equivalently, residuated monoid homomorphisms. 

The ring-like countenance of quantales obviously suggests a natural definition of module. Given a quantale $Q$, a \emph{left module over $Q$} (or, simply, {\em left $Q$-module}) is a sup-lattice $(M, \bigvee)$ acted on by $Q$ via a \emph{scalar multiplication} $\cdot: (a,u) \in Q \times M \mapsto a \cdot u \in M$ such that  
\begin{itemize}
\item $(ab) \cdot u = a \cdot (b \cdot u)$, for all $a, b \in Q$ and $u \in M$;
\item the scalar multiplication distributes over arbitrary joins in both arguments or, equivalently, is biresiduated;
\item $1 \cdot u = u$, for all $u \in M$.\footnote{Using a different symbol for this action would make the notations much heavier without helping the reading, so we rather preferred to use the same symbol of the product in the quantale, relying on the context and different sets of letters for scalars and ``vectors'' for the meaning of each of its occurrences. Whenever convenient, we shall also drop it.}
\end{itemize}
Right modules are defined analogously, mutatis mutandis. Moreover, if $R$ is another quantale, a sup-lattice $M$ is a $Q$-$R$-bimodule if it is a left $Q$-module, a right $R$-module, and in addition $(a \cdot_Q u) \cdot_R a' = a \cdot_Q (u \cdot_R a')$ for all $a \in Q$, $a' \in R$, and $u \in M$.

We also recall that, as for the quantale product, the biresiduation of $\cdot$ induces two more maps:
\begin{itemize}
\item[] $\under: (a,u) \in Q \times M \mapsto a\under u = \max\{v \in M \mid av \leq u\} \in M$, and
\item[] $/: (u,v) \in M \times M \mapsto u/v = \max\{a \in Q \mid av \leq u\} \in Q.$
\end{itemize}
We shall normally refer to ``modules'' and use the left module notation whenever a definition or a result can be stated both for left and right modules.

Given two $Q$-modules $M$ and $N$, a map $f: M \to N$ is a $Q$-module homomorphism if it is a sup-lattice homomorphism which preserves the scalar multiplication, namely, an action-preserving residuated map. For any quantale $Q$ we shall denote by $\mc$ and $\MC$ respectively the categories of left $Q$-modules and right $Q$-modules with the corresponding homomorphisms. Moreover, if $R$ is another quantale $\mc_R$ shall denote the category whose objects are $Q$-$R$-bimodules and morphisms are maps which are simultaneously left $Q$-module and right $R$-module morphisms.





We recall that a \emph{bifunctor} (short for binary functor) or \emph{functor of two variables} is simply a functor whose domain is the product of two categories. For $\cat C_1$, $\cat C_2$ and $\cat D$ categories, a functor $F: \cat C_1 \times \cat C_2 \rightarrow \cat D$ is a \emph{bifunctor} from $\cat C_1$ and $\cat C_2$ to $\cat D$.

Let $\cat C$ be a category and $\cat J$ be a (small) category. A functor $X:\cat J \rightarrow \cat C$ is also called a (resp.: small) $\cat C$-\emph{diagram} of shape $\cat J$.

\begin{definition}
Let $F:\cat J \rightarrow \cat C$ be a diagram of shape $\cat J$ in a category $\cat C$. A \textit{cone} to $F$ is an object $N$ of $\cat C$ together with a family $\psi_X:N \rightarrow F(X)$ of morphisms indexed by the objects $X$ of $\cat J$, such that for every morphism $f:X \rightarrow Y$ in $\cat J$, we have $F(f)\circ \psi_X = \psi_Y$.

A \textit{limit} of the diagram $F:\cat J \rightarrow \cat C$ is a cone $(L,\phi)$ to $F$ such that for every other cone $(N,\psi)$ to $F$ there exists a unique morphism $u:N \rightarrow L$ such that $\phi_X \circ u = \psi_X$ for all $X$ in $\cat J$.

One says that the cone $(N,\psi)$ factors through the cone $(L,\phi)$ with the unique factorization $u$. The morphism $u$ is sometimes called the \emph{mediating morphism}.

The dual notions of limits and cones are colimits and co-cones. We can obtain their definitions by inverting all morphisms in the above definitions. A \textit{co-cone} of a diagram $F: \cat J \rightarrow \cat C$ is an object $N$ of $\cat C$ together with a family of morphisms $\psi_X:F(X) \rightarrow N$ for every  objects $X$ of $\cat J$, such that for every morphism $f:X \rightarrow Y$ in $\cat J$, we have $\psi_Y \circ F(f) = \psi_X$. A \textit{colimit} of the diagram $F: \cat J \rightarrow \cat C$ is a co-cone $(L,\phi)$ of $F$ such that for any other co-cone $(N,\psi)$ of $F$ there exists a unique morphism $u:L \rightarrow N$ such that $u \circ \phi_X = \psi_X$ for all $X$ in $\cat J$.
\end{definition}

A category $\cat C$ is \emph{complete} if it has all small limits, that is, every small diagram has a limit in $\cat C$. It is \emph{cocomplete} if it has all small colimits, that is, every small diagram has a colimit in $\cat C$. A category $\cat C$ is cocomplete if and only if its opposite category is complete.

\section{Tensor product of quantale modules}
\label{tensor}

The existence of the tensor product of quantale modules was established in \cite[Theorem 6.3]{ruscorrapal}. For the reader's convenience, we recall the statement of the theorem; for more details on the construction of the tensor product of quantale modules the reader may refer to the cited paper.
\begin{theorem}\label{tensormqexists}
Let $M_1$ be a right $Q$-module and $M_2$ a left $Q$-module. Then the tensor product $M_1 \tensor_Q M_2$ of the $Q$-modules $M_1$ and $M_2$ exists. It is, up to isomorphisms, the quotient $\mathscr{P}(M_1 \times M_2)/\vartheta_R$ of the free sup-lattice generated by $M_1 \times M_2$ with respect to the (sup-lattice) congruence relation generated by the set
\begin{equation}\label{R}
\rho = \left\{
	\begin{array}{l}
		\left(\left\{\left(\bigvee X, y\right)\right\}, \bigcup_{x \in X}\{(x,y)\}\right) \\
		\left(\left\{\left(x, \bigvee Y\right)\right\}, \bigcup_{y \in Y}\{(x,y)\}\right) \\
		\left(\{(x \cdot_1 a, y)\}, \{(x,a \cdot_2 y)\}\right) \\
	\end{array} \right\vert
	\left.
	\begin{array}{l}
	X \subseteq M_1, y \in M_2 \\
	Y \subseteq M_2, x \in M_1 \\
	a \in Q \\
	\end{array}
	\right\}.
\end{equation}
\end{theorem}
In \cite{rusapal}, it was also shown that each quantale morphism $h: Q \to R$ induces an adjoint and co-adjoint functor $( \ )_h: \mc[R] \to \mc$, whose left adjoint is $R \tensor_Q \underline{\ \ }$. 

In this section, we shall establish various properties of the tensor product of quantale modules, which will be useful in the next sections.

\begin{proposition}
Let $M$ be a $Q$-$R$-bimodule and let $N$ be an $R$-$S$-bimodule, and let $f \in \mc_R (M_1,M)$ and $g \in \mc[R]_S (N_1,N)$. So the assignments $(M,N) \mapsto M \tensor N$ and $(f,g) \mapsto f \tensor g$ define a bifunctor $- \tensor -: \mc_R \times \mc[R]_S \rightarrow \mc_S$.
Furthermore let $\mm[N][R]{S}$ and $\mm[O][Q]{S}$ be bimodules, then on $\MC[S] (N,O)$ there is a $Q$-$R$-bimodule structure defined by
$$(qgr)(n)=q(g(rn)) \hspace{1cm}\mbox{ for } n \in N, \, g \in \MC[S] (N,O),$$
and if $\QMR$ and $\mm[O][Q]{S}$ are bimodules, then on $\mc (M,O)$ there is the $R$-$S$-bimodule structure defined by
$$(rfs)(m)=f(mr)s \hspace{1cm}\mbox{ for } m \in M, \, f \in \mc (M,O).$$
\end{proposition}
\begin{proof}
Let $(f,g): (M,N) \rightarrow (M_1,N_1)$, with $M,M_1 \in \MC$, $N,N_1 \in \mc$. We shall define $f\tensor g: M\tensor N \rightarrow M_1 \tensor N_1$ e $h\tensor k: M_1\tensor N_1 \rightarrow M_2 \tensor N_2$, and prove that $- \tensor -$ preserves composition. 

Let $\varphi: (x,y) \in M\times N \mapsto f(x)\tensor g(y) \in M_1\tensor N_1$; we have that
$$\begin{array}{ll}
\varphi(xq,y) & = f(xq)\tensor g(y) \\ & = f(x)q\tensor g(y) \\ & = f(x)\tensor qg(y) \\ & = f(x)\tensor g(qy) \\ & = \varphi(x,qy).
\end{array}$$
We also have
$$\begin{array}{ll}
\varphi\left(\bigvee X,y\right) & = f\left(\bigvee X\right)\tensor g(y) \\ & = \left(\bigvee_{x \in X} f(x)\right)\tensor g(y) \\ & = \bigvee_{x \in X} (f(x)\tensor g(y)) \\ & = \bigvee_{x \in X} \varphi(x,y),
\end{array}$$
and, similarly, $\varphi \left( x,\bigvee Y\right) = \bigvee_{y \in Y} \varphi(x,y)$. So, $\varphi$ is a $Q$-bimorphism and, therefore, it canonically induces a sup-lattice morphism $f\tensor g: M \tensor N \to M_1 \tensor N_1$. Analogously, we can define sup-lattice morphisms $h \tensor k: M_1\tensor N_1 \rightarrow M_2 \tensor N_2$ and $(h\circ f)\tensor(k\circ g): M \tensor N \to M_2 \tensor N_2$. Furthermore, we have $(h\circ f)\tensor(k\circ g) = (h\tensor k)\circ(f\tensor g)$ because $h(f(x))\tensor k(g(y)) = (h\tensor k)(f(x)\tensor g(y)) = ((h\tensor k)\circ(f\tensor g))(x\tensor y)$.

It follows that $- \tensor -$ is a bifunctor. The second part of the proposition is obvious.
\end{proof}

By routine calculations (see, for example, \cite[Proposition 3.2]{kat7}), the following can be easily proved.
\begin{proposition}
Let $N$, $N_1$ be $Q$-$R$-bimodules, $O$, $O_1$ be $S$-$R$-bimodules, $\beta \in \mc_R (N_1,N)$, and $\gamma \in \mc[S]_R (O,O_1)$. The assignments $(N,O) \mapsto \MC[R] (N,O)$ and
$$(\beta,\gamma) \mapsto \mc[S]_Q(\beta,\gamma): f \in \MC[R] (N,O) \mapsto \gamma \circ f \circ \beta \in \MC[R] (N_1,O_1)$$
define a bifunctor $\MC[R]: ({\mc_R})^{op} \times \mc[S]_R \rightarrow \mc[S]_Q$.

Analogously, for $M$, $M_1$ $S$-$Q$-bimodules, and $O$, $O_1$ $S$-$R$-bimodules, the assignments $(M,O) \mapsto \mc[S] (M,O)$ and
$$(\alpha,\gamma) \mapsto \mc[Q]_R(\alpha,\gamma): f \in \mc[S] (M,O) \rightarrow \gamma \circ f \circ \alpha \in \mc[S] (M_1,O_1),$$
where $\alpha \in \mc[S]_Q (M_1,M)$ and $\gamma \in \mc[S]_R (O,O_1)$, define a bifunctor $\mc[S]: ({\mc[S]_Q})^{op} \times \mc[S]_R \rightarrow \mc_R$.
\end{proposition}


The following proposition, which is the quantale-theoretic analogous of \cite[Theorem 3.3]{kat7}, is an immediate consequence of \cite[Lemma 6.5]{rusapal}. 

\begin{proposition}\label{natiso}
Let $Q$, $R$ e $S$ be quantales, and let $\mm[M][S]{Q}$, $\mm[N][Q]{R}$ and $\mm[O][S]{R}$ be bimodules. Then there are natural isomorphisms $$\varphi : \mc[S]_R(M \tensor_Q N, O) \rightarrow \mc[S]_Q(M,\MC[R](N,O))$$ and $$\psi : \mc[S]_R(M \tensor_Q N, O) \rightarrow \mc[Q]_R(N,\mc[S](M,O)).$$
In particular, for any bimodule $\mm[N][Q]{R}$, the functor $\MC[R](N,-) : \mc[S]_R \rightarrow \mc[S]_Q$ is a right adjoint to the functor $-\tensor_Q N : \mc[S]_Q \rightarrow \mc[S]_R$, i.e., $-\tensor_Q N \dashv \MC[R](N,-)$; and for any bimodule $\mm[M][S]{Q}$, the functor $\mc[S](M,-) : \mc[S]_R \rightarrow \mc_R$ is a right adjoint to the functor $M \tensor_Q - : \mc_R \rightarrow \mc[S]_R$. i.e. $M \tensor_Q - \dashv \mc[S](M,-)$.
\end{proposition}

\begin{proposition}\label{natisom}
Let $Q$, $R$, $S$ and $T$ be quantales, and let $\mm[M][S]{Q}$, $\mm[N][Q]{R}$ and $\mm[O][R]{T}$ be bimodules. Then there is a natural isomorphism
$$(M \tensor_Q N) \tensor_R O \cong M \tensor_Q (N \tensor_R O).$$
\end{proposition}
\begin{proof}
For all $x \in M$, let us define
$$\alpha_x: (y,z) \in N \times O \mapsto (x\tensor y)\tensor z \in (M\tensor N)\tensor O,$$
and show that it is a bimorphism. Indeed, we have:
$$\begin{array}{ll}
\alpha_x(yr,z) & = (x\tensor yr)\tensor z \\ & = (x\tensor y)r\tensor z \\ & = (x\tensor y)\tensor rz \\ & = \alpha_x(y,rz),
\end{array}$$
$$\begin{array}{ll}
\alpha_x\left(\bigvee Y,z\right) & = \left( x\tensor \bigvee Y\right)\tensor z \\ & = \left(x\tensor \bigvee\limits_{y \in Y} y\right)\tensor z \\ & = \left(\bigvee\limits_{y \in Y} (x\tensor y)\right)\tensor z \\ & = \bigvee\limits_{y \in Y} ((x\tensor y)\tensor z) \\ & = \bigvee\limits_{y \in Y} \alpha_x(y,z),
\end{array}$$
and, similarly,
$$\alpha_x\left(y,\bigvee Z\right) = \bigvee\limits_{z \in Z} \alpha_x(y,z).$$
So, there exists a unique sup-lattice morphism $\overline{\alpha}_x: N \tensor O \rightarrow (M\tensor N)\tensor O$ extending $\alpha_x$; observe that, for all $x \in M$, $y \in N$, and $z \in O$, $\ov\alpha_x(y\tensor z) = (x\tensor y)\tensor z$.

Now, let
$$\alpha: \left(x,\bigvee_{i \in I}y_i\tensor z_i\right) \in M \times (N \tensor O) \mapsto \overline{\alpha}_x\left(\bigvee_{i \in I}y_i \tensor z_i\right) \in (M\tensor N)\tensor O,$$
which is a bimorphism too. Indeed, we have
$$\begin{array}{ll}
\alpha\left(\bigvee X,\bigvee\limits_{i \in I}y_i \tensor z_i\right) & = \overline{\alpha}_{\bigvee X}\left(\bigvee\limits_{i \in I}y_i \tensor z_i\right) \\ & = \bigvee\limits_{i \in I}\overline{\alpha}_{\bigvee X}(y_i \tensor z_i) \\ & =\bigvee\limits_{i \in I}\alpha_{\bigvee X}(y_i , z_i) \\ & = \bigvee\limits_{i \in I}\left( \bigvee X\tensor y_i\right)\tensor z_i \\ & = \bigvee\limits_{i \in I}\bigvee\limits_{x \in X}(x\tensor y_i)\tensor z_i \\ & = \bigvee\limits_{x \in X}\bigvee\limits_{i \in I}(x\tensor y_i)\tensor z_i \\ & = \bigvee\limits_{x \in X}\bigvee\limits_{i \in I}\overline{\alpha}_x(y_i\tensor z_i) \\ & = \bigvee\limits_{x \in X}\overline{\alpha}_x\left(\bigvee\limits_{i \in I}y_i\tensor z_i\right) \\ & = \bigvee\limits_{x\in X}\alpha\left(x,\bigvee\limits_{i \in I}y_i \tensor z_i\right).
\end{array}$$

So, by the universal property of the tensor product $M \otimes_Q (N \otimes_R O)$, there exists a unique sup-lattice morphism
$$\overline{\alpha}: M \otimes_Q (N \otimes_R O) \longrightarrow (M \otimes_Q N)\otimes_R O$$
such that, for all $x \in M$, $y \in N$, and $z \in O$, we have $\overline{\alpha}\bigl(x \otimes (y \otimes z)\bigr) = (x \otimes y) \otimes z$.

Symmetrically, for each $z \in O$ let us define $\beta_z : M \times N \longrightarrow M \otimes_Q (N \otimes_R O)$, $\beta_z(x,y) = x \otimes (y \otimes z)$. As before, $\beta_z$ is a bimorphism, so there exists a unique sup-lattice morphism
$$\overline{\beta}_z : M \otimes_Q N \longrightarrow M \otimes_Q (N \otimes_R O)$$
such that $\overline{\beta}_z(x \otimes y) = x \otimes (y \otimes z)$ for all $x \in M$, $y \in N$.  

Now define $\beta : (M \otimes_Q N) \times O \longrightarrow M \otimes_Q (N \otimes_R O)$ by 
$$\beta\!\left(\bigvee_j x_j \otimes y_j, z\right) = \bigvee_j \overline{\beta}_z(x_j \otimes y_j).$$
Again, $\beta$ is a bimorphism, so by the universal property of the tensor product $(M \otimes_Q N)\otimes_R O$, there exists a unique sup-lattice morphism
$$\overline{\beta} : (M \otimes_Q N)\otimes_R O \longrightarrow M \otimes_Q (N \otimes_R O)$$
such that, for all $x \in M$, $y \in N$, $z \in O$, $\overline{\beta}\bigl((x \otimes y) \otimes z\bigr) = x \otimes (y \otimes z)$.

So, 
$$\overline{\beta} \circ \overline{\alpha}\bigl(x \otimes (y \otimes z)\bigr) = \overline{\beta}\bigl((x \otimes y)\otimes z\bigr) = x \otimes (y \otimes z), \text{ and}$$
$$\overline{\alpha} \circ \overline{\beta}\bigl((x \otimes y)\otimes z\bigr)
= \overline{\alpha}\bigl(x \otimes (y \otimes z)\bigr)
= (x \otimes y) \otimes z.$$

Since both morphisms preserve arbitrary joins and the tensors generate the respective tensor products, we conclude that

$\overline{\beta} \circ \overline{\alpha} = \mathrm{id}_{M \otimes_Q (N \otimes_R O)} \quad\text{and}\quad \overline{\alpha} \circ \overline{\beta} = \mathrm{id}_{(M \otimes_Q N)\otimes_R O}$.

Therefore, $\overline{\alpha}$ is an isomorphism with inverse $\overline{\beta}$.

\end{proof}

\begin{proposition}
$E:=\End(\mm{})$ is a quantale and $M$ is a $Q$-$E$-bimodule.
\end{proposition}
\begin{proof}
We note  that $\la \End_Q(M), \bigvee \ra$ is a sup-lattice, and $\la \End_Q(M), \cdot, \id_M \ra$ is a monoid, where the product $h\cdot g$ is the composition $g\circ h$.

Let $h \in \End_Q(M)$ and $\{g_i\}_{i \in I} \subseteq \End_Q(M)$. Note that
$$\begin{array}{ll}
\left(\left(\bigvee\limits_{i \in I} g_i\right)\cdot h\right)(x) & = \left(h \circ \left(\bigvee\limits_{i \in I} g_i\right)\right) (x) \\ & = h\left(\left(\bigvee\limits_{i \in I} g_i\right)(x)\right) \\ & = h\left(\bigvee\limits_{i \in I} g_i(x)\right) \\ & = \bigvee\limits_{i \in I} h(g_i(x)) \\ & = \bigvee\limits_{i \in I} (h\circ g_i)(x) \\ & = \bigvee\limits_{i \in I} (g_i\cdot h)(x).
\end{array}$$
Similarly, we have $\left( h\cdot\left(\bigvee_{i \in I} g_i\right)\right)(x) = \bigvee_{i \in I} (h\cdot g_i )(x)$, whence $E=\End_Q(M)$ is a quantale.

Furthermore, we have that $M$ is a right $E$-module. Indeed, taking $xh = h(x)$, $x \in M$, $h \in E$, 
$$x(f\cdot g) =  (f\cdot g)(x) = (g\circ f)(x) = g(f(x)) = f(x)g=  (xf)g.$$
We have also, for $x \in M$, $f \in E$, $F \subseteq E$ and $N \subseteq M$,
$$x\left(\bigvee F\right) = \left(\bigvee F\right)(x) = \left(\bigvee_{f \in F} f \right)(x) = \bigvee_{f \in F} f(x) = \bigvee_{f \in F} xf $$
and
$$\left(\bigvee N\right)f = f\left(\bigvee N\right) = f\left(\bigvee_{x \in N} x \right) = \bigvee_{x \in N} f(x) = \bigvee_{x \in N} xf.$$
It is clear that $x\id_E = \id_E(x) = x$.
\end{proof}




\section{Projective generators}
\label{progenerator}

As we are going to see, projective generators of quantale module categories play a prominent role in the characterization of Morita equivalent quantales. Projective modules have been characterized in \cite{russajl}; in what follows, we shall characterize generators of the categories of quantale modules.

\begin{lemma}
For every endomorphism $\alpha: \mm[Q]{} \rightarrow \mm[Q]{}$ of the left $Q$-module $Q$ there exists $r \in Q$ such that $\alpha(q) = qr$ for all $q \in Q$. 
\end{lemma}

\begin{proof}
Indeed, if $\alpha(1) := r \in Q$, then $\alpha(x) = \alpha(x1) = x\alpha(1) = xr$, for any $x \in Q$.
\end{proof}

\begin{proposition}\label{projchar}
A $Q$-module $\mm{}$ is projective if and only if it is isomorphic to the image of some free module $Q^X$ under an idempotent endomorphism. In particular, a cyclic $Q$-module $M$ is projective if and only if it is isomorphic to $Q \cdot u$ for some multiplicatively idempotent $u \in Q$.
\end{proposition}
\begin{proof}
Follows readily from \cite[Theorem 2.4]{russajl}.
\end{proof}

Let $E:=\End(\mm{})$ be the quantale of all endomorphisms of a $Q$-module $\mm{}$. So, considering endomorphisms of $\mm{}$ operating on the right of $\mm{}$, it is easy to see that $M$ becomes a $Q$-$E$-bimodule, i.e., $\mm{E}$. We denote by $N = M^* := \mc(\mm{},\mm[Q]{})$ the \textit{dual} $E$-$Q$-bimodule $N=M^*$ of the $Q$-$E$-bimodule $M$, where the $E$-$Q$-bimodule structure on $N$ is defined by the rule $$m(hfq) = ((mh)f)q$$ for given $m \in M$, $h \in E$, $f \in N$ and $q \in Q$, where the elements of both $E$ and $N$ act on the right of $M$. In particular, we have $m(fq) = (mf)q$ and $m(hf) = (mh)f$
(by “$MNQ$-associativity” and “$MEN$-associativity”, respectively (see \cite[Section 18C]{lam})). Here and in the sequel, the elements $m, f, q, h$ (and $m', f', q', h'$) are in $M, N, Q$ and $E$, respectively. By means of the “$MNM$-associativity”, we also define the endomorphism $fm \in E$ by $m'(fm) = (m'f)m$ for any $m \in M$ and $f \in N$. Then, in the same way as it has been done in \cite[Section 18C]{lam}, one can show that $fm$ is indeed an endomorphism of $\mm{}$ and $(f'm)f = f'(mf)$ for any $m \in M, f, f' \in N$ (“$NMN$-associativity”), as well as the following result.

\begin{lemma}
The assignments $(m,f) \mapsto mf$ and $(f,m) \mapsto fm$ define the $(Q,Q)$-module homomorphism $\alpha: M \tensor_E M^* \rightarrow Q$ and the $(E,E)$-homomorphism $\beta: M^* \tensor_Q M \rightarrow E$, respectively.
\end{lemma}

\begin{proof}

Let us define the maps

$$\begin{array}{ll}
(h,f) \in E \times M^* \mapsto hf = f \circ h \in M^*, & \text{ and}\\
(f,q) \in  M^* \times Q \mapsto fq \in M^*, & \text{ where }fq :x \in M \mapsto f(x)q \in Q.
\end{array}$$

Next, we set
$$\alpha': (m,f) \in M \times M^* \mapsto f(m) \in Q, \text{and show that it is a bimorphism.}$$

Indeed, since $M$ is a right $E$-module, taking $mh = h(m)$, $m \in M$, $h \in E$, we have
$$\begin{array}{l}
\alpha'(mh,f) = \alpha'(h(m),f) = f(h(m)) = f \circ h(m) = hf(m) = \alpha'(m,hf).
\end{array}$$
Furthermore,
$$\begin{array}{l} 
\alpha'\left(\bigvee X,f\right) = f\left( \bigvee X\right) = \bigvee\limits_{x \in X} f(x) =\bigvee\limits_{x \in X} (\alpha'(x,f)) \text{ and }\\
\alpha'\left(x,\bigvee F\right) = \bigvee F (x) = \bigvee\limits_{f \in F} f(x) = \bigvee\limits_{f \in F} (\alpha'(x,f)).
\end{array}$$

So, there exists a unique sup-lattice morphism $\alpha:M \tensor M^* \to Q$ extending $\alpha'$, where, for all $m \in M$ and $f \in M^*$, $(m\tensor f) \in M \tensor M^* \mapsto f(m) \in Q$.

Note also that $\alpha$ is a $Q$-$Q$-homomorphism, because

$$\begin{array}{ll}
\alpha\left(q\left(\bigvee\limits_{i \in I}x_i\tensor f_i\right)\right) & = \alpha\left(\bigvee\limits_{i \in I}qx_i\tensor f_i\right) \\ & = \bigvee\limits_{i \in I}\alpha(qx_i\tensor f_i) \\ & = \bigvee\limits_{i \in I}f_i(qx_i) \\ & = \bigvee\limits_{i \in I}qf_i(x_i) \\ & = q\bigvee\limits_{i \in I}f_i(x_i) \\ & = q \bigvee\limits_{i \in I}\alpha(x_i\tensor f_i) \\ & = q\alpha \left(\bigvee\limits_{i \in I}x_i\tensor f_i\right),
\end{array}$$
and
$$\begin{array}{ll}
\alpha\left(\left(\bigvee\limits_{i \in I}x_i\tensor f_i\right)q\right) & = \alpha\left(\bigvee\limits_{i \in I}x_i\tensor f_iq\right) \\ & = \bigvee\limits_{i \in I}\alpha(x_i\tensor f_iq) \\ & = \bigvee\limits_{i \in I}f_iq(x_i) \\ & = \bigvee\limits_{i \in I}f_i(x_i)q \\ & = \left(\bigvee\limits_{i \in I}f_i(x_i)\right)q \\ & = \left(\bigvee\limits_{i \in I}\alpha(x_i\tensor f_i)\right)q \\ & = \alpha \left(\bigvee\limits_{i \in I}(x_i\tensor f_i)\right)q.
\end{array}$$

Now consider the map 

$$\beta': (f,m) \in M^* \times M \mapsto fm \in E, \text{and show that it is a bimorphism.}$$

Indeed, since $M^*$ is a right $Q$-module, taking $fq = f(x)q$, $x \in M$, $q \in Q$, we have
$$\begin{array}{l}
\beta'(fq,m) = (fq)m = (f(x)q)m = f(x)(qm) = f\cdot(qm) = \beta'(f,qm).
\end{array}$$
Furthermore,
$$\begin{array}{l} 
\beta'\left(\bigvee F,m\right) = \left(\bigvee F\right)m =\left(\bigvee\limits_{f \in F} f\right) m = \left(\bigvee\limits_{f \in F} fm\right) =\bigvee\limits_{f \in F} \beta'(f,m), \\ \text{and }\\ \\
\beta'\left(f,\bigvee Y\right) =  f \bigvee Y = \bigvee\limits_{m \in Y} fm =\bigvee\limits_{m \in Y} \beta'(f,m). 
\end{array}$$

So, there exists a unique sup-lattice morphism $\beta:M^* \tensor M \to E$ extending $\beta'$, where, for all $m \in M$ and $f \in M^*$, $(f\tensor m) \in M^* \tensor M \mapsto fm \in E$.

We have also that $\beta$ is a $E$-$E$-homomorphism, because

$$\begin{array}{ll}
\beta\left(h\left(\bigvee\limits_{i \in I}f_i\tensor m_i\right)\right)(x) & = \beta\left(\bigvee\limits_{i \in I}hf_i\tensor m_i\right)(x) \\ & = \bigvee\limits_{i \in I}\beta(hf_i\tensor m_i)(x) \\ & = \bigvee\limits_{i \in I}(hf_i)(x)m_i \\ & = \bigvee\limits_{i \in I}(f_i\circ h)(x)m_i \\ & = \bigvee\limits_{i \in I}f_i (h(x))m_i \\ & = \bigvee\limits_{i \in I}\beta(f_i\tensor m_i)(h(x)) \\ & = \bigvee\limits_{i \in I}(\beta(f_i\tensor m_i)\circ h)(x) \\ & = \bigvee\limits_{i \in I}h\beta(f_i\tensor m_i)(x) \\ & = h \left(\bigvee\limits_{i \in I}\beta(f_i\tensor m_i)\right)(x) \\ & = h\beta \left(\bigvee\limits_{i \in I}f_i\tensor m_i\right)(x)
\end{array}$$
and
$$\begin{array}{ll}
\beta\left(\left(\bigvee\limits_{i \in I}f_i\tensor m_i\right)h\right)(x) & = \beta\left(\bigvee\limits_{i \in I}f_i\tensor m_ih\right)(x) \\ & = \bigvee\limits_{i \in I}\beta(f_i\tensor m_ih)(x) \\ & = \bigvee\limits_{i \in I}\beta(f_i\tensor h(m_i))(x) \\ & = \bigvee\limits_{i \in I}f_i(x)h(m_i) \\ & = \bigvee\limits_{i \in I}f_i(x)h(m_i) \\ & = \bigvee\limits_{i \in I}h(f_i(x)(m_i)) \\ & = \left(\bigvee\limits_{i \in I}h(\beta(f_i\tensor m_i))(x)\right) \\ & = \left(\bigvee\limits_{i \in I}h \circ (\beta f_i\tensor m_i))(x)\right) \\ & =  \left(\bigvee\limits_{i \in I}(\beta(f_i\tensor m_i)h)(x)\right) \\ & = \left(\beta\left(\bigvee\limits_{i \in I}(f_i\tensor m_i)\right)h\right)(x).
\end{array}$$

\end{proof}

Let us now consider the dual left $Q$-module of a free left module $\mm[Q^n][Q]{}$ and consider
the standard generating set $\{e_1, \ldots, e_n\}$ of  $\mm[Q^n][Q]{}$, namely the set of $n$-tuples $e_i$, $i= 1, \ldots, n$, with all components equal to $\bot$ except for the $i$-th, which is equal to $1$. Now let, for all $i = 1, \ldots, n$, $e_i^{*} \in (\mm[Q^n][Q]{})^{*} = \mc (\mm[Q^n][Q]{},\mm[Q][Q]{})$ be defined by
\[
e_j e_i^{*} :=
\begin{cases}
1, & \text{if } i = j, \\
\bot, & \text{if } i \neq j.
\end{cases}
\]
It is easy to see that $\{e_1^{*}, \ldots, e_n^{*}\}$ is a minimal generating set for $(\mm[Q^n][Q]{})^{*}$, called (with a slight abuse) a \emph{dual basis}, of $(\mm[Q^n][Q]{})^{*}$. 

Furthermore, it is easy to see that an analogous definition can be given for the right dual $Q$-module $(\mm[Q^n][Q]{})^{*}$.

\begin{proposition}\label{betaiso}
Let $\mm{}$ be a left $Q$-module, and let us consider its dual module $M^*:=\mc (M,Q)$ and the quantale $E:=\End(\mm{})$. The following are equivalent.
\begin{enumerate}
\item[](a) $M$ is projective;
\item[](b) $\beta$ is an isomorphism;
\item[](c) $\beta$ is a surjective bimodule homomorphism.
\end{enumerate}   
\end{proposition}
\begin{proof}


$(a) \Rightarrow (b)$

Let $M$ be a projective module, and let us identify $M$ with its isomorphic submodule of some free module $\mm[Q^I][Q]{}$, according to Proposition \ref{projchar}.  Then there exists an endomorphism $e: \mm[Q^I][Q]{} \rightarrow \mm[Q^I][Q]{}$ such that the restriction of $e$ to the submodule $M$ is the identity, i.~e. $me=m, \forall m$. Let $\overline{e_i^*}$ be the restriction to $M$ of the composition $\mm[Q^I][Q]{} \stackrel{e}{\longrightarrow} \mm[Q^I][Q]{} \stackrel{e_i^*}{\longrightarrow} \mm[Q][Q]{}$, for each $i \in I$. It is clear that $x=\bigvee(xe_i^*)e_i$ for any element $x=\bigvee x_ie_i \in \mm[Q^I][Q]{}$, hence, applying the ``MEM-associativity'', one has $m=me=\bigvee(mee_i^*)e_ie=\bigvee(m\overline{e_i^*})(e_ie)=\bigvee m\overline{e_i^*}(e_ie)=m \bigvee\overline{e_i^*}(e_ie)$ $\forall m \in M$, with $e_ie \in M$. Therefore $\bigvee\overline{e_i^*}(e_ie)=\beta(\bigvee\overline{e_i^*}\tensor(e_ie))=1_M \in E$, and, since $\beta$ in an $(E,E)$-homomorphism, $\beta$ is a surjection.

Suppose now that
$$\beta\left(\bigvee\limits_{i,j \in I} e_i^*a_i\tensor b_je_j\right) = \beta\left(\bigvee\limits_{i,j \in I} e_i^*a_i'\tensor b_j'e_j\right), \hspace{0.5cm} a_i, a_i',b_j, b_j' \in Q.$$
Then, for any $e_i$, we have
$$\begin{array}{ll}
a_i\bigvee\limits_{j \in I} b_je_j & = \bigvee\limits_{j \in I} a_i b_je_j \\ & =  \bigvee\limits_{j \in I} e_ie_i^*a_i b_je_j \\ & = e_i\beta\left(\bigvee\limits_{i,j \in I} e_i^*\tensor a_i b_je_j\right) \\ & = e_i\beta\left(\bigvee\limits_{i,j \in I} e_i^*a_i\tensor b_je_j\right) \\ & =  e_i\beta\left(\bigvee\limits_{i,j \in I} e_i^*a_i'\tensor b_j'e_j\right) \\ & = e_i\beta\left(\bigvee\limits_{i,j \in I} e_i^*\tensor a_i' b_j'e_j\right) \\ & =\bigvee\limits_{j \in I} e_ie_i^*a_i' b_j'e_j \\ & = \bigvee\limits_{j \in I} a_i' b_j'e_j \\ & = a_i'\bigvee\limits_{j \in I} b_j'e_j.
\end{array}$$
Therefore,
$$\begin{array}{ll}
\bigvee\limits_{i,j \in I} e_i^*a_i\tensor b_je_j & = \bigvee\limits_{i,j \in I} e_i^*\tensor a_i b_je_j \\ & = \bigvee\limits_{i\in I} e_i^*\left(a_i \bigvee_{j \in I} b_je_j\right) \\ & = \bigvee\limits_{i\in I} e_i^*\left(a_i' \bigvee_{j \in I} b_j'e_j\right) \\ & = \bigvee\limits_{i,j \in I} e_i^*\tensor a_i' b_j'e_j \\ & = \bigvee\limits_{i,j \in I} e_i^*a_i'\tensor b_j'e_j,
\end{array}$$
i.e., $\beta$ is injective. Therefore, $\beta$ is an isomorphism. 

\item $(b) \Rightarrow (c)$ is obvious. 

\item $(c) \Rightarrow (a)$ 

Suppose that there are $n_i \in M^*$ and $m_j \in M$ $i \in I$ such that $\bigvee n_im_i = \beta\left(\bigvee\limits_{i \in I} n_i\tensor m_i\right) = 1_M \in E$. 

Then, given the surjection $\theta:\mm[Q^I][Q]{} \twoheadrightarrow  \QM$ and the injection $\mu: \QM \rightarrow \mm[Q^I][Q]{}$ respectively defined by the assignments $e_i \mapsto m_i$ and $m \mapsto \bigvee mn_ie_i$, we get that $\theta \mu = 1_M$. So $M$ is a retract of a free module, therefore projective.


\end{proof}

\begin{definition}
An object $S$ of a category $\mathcal{C}$ is called a \emph{separator} if for every pair of morphisms $f,g:X \rightarrow Y \in \mathcal{C}$, if $f\circ e=g\circ e$ for every morphism $e:S\rightarrow X$, then $f=g$, or equivalently, if $f \neq g$, then there is a morphism $e:S\rightarrow X$ such that $f\circ e \neq g\circ e$.
\end {definition}

\begin{definition}
A $Q$-module $M$ is said to be a \emph{generator} for $\mc$ if the free module $Q$ is a retract of $M^I$ for some set $I$.
\end{definition}

\begin{definition}
The \emph{trace} of a $Q$-module $M$ is defined as
$$\operatorname{tr}(M) :=\left\langle\bigcup\limits_{f \in M^*}fM\right\rangle.$$
\end{definition}


\begin{proposition}\label{generchar}
Let $Q$ be a quantale and $M$ a $Q$-module. Then $\operatorname{tr}(M) = Q$ if and only if $M$ is a generator for $\mc$.
\end{proposition}

\begin{proof}

\begin{enumerate}
 
\item[($\Rightarrow$)]

Let $\operatorname{tr}(M) = Q$, i.e., $Q=\left\langle  f M: f \in M^* \right\rangle$. For all $f \in M^*$, let us consider the map 

$$\begin{array}{rcl}
\alpha_f:M & \longrightarrow & M^{M^*}\\
m & \,\longmapsto & \alpha_f(m)=(m_h^f)_{h\in M^*} = \begin{cases}
m_f^f=m; \\
m_h^f=\bot, & \text{if }\ h \neq f.
\end{cases}
\end{array}$$

The function $$\begin{array}{cccc}
g: & M^{M^*} &  \longrightarrow & Q\\
&(m_h)_{h\in M^*} & \,\longmapsto & \dps{\bigvee_{h\in M^*}h(m_h)}
\end{array}$$  
is obviously a $Q$-module morphism. For all $f \in M^*$ and for all $m \in M$, we have $g(\alpha_f(m))=g((m_h^f)_{h\in M^*})= \bigvee\limits_{h\in M^*}h(m_h^f)=f(m)$.

Therefore $g$ makes the following diagram commute for each $f \in M^*$:

$$\begin{tikzcd}
{M^{M^*}} \arrow[rrd, "g"]                                  &  &    \\
                                                     &  & {Q} \\
{M} \arrow[rru, "f"'] \arrow[uu, "\alpha_f", hook'] &  &   
\end{tikzcd}.$$
Since $\operatorname{tr}(M) = Q$, for all $q \in Q$ there exist $(m_h)_{h\in M^*} \in M^{M^*}$ and $(q_h)_{h\in M^*} \in Q^{M^*}$ such that
$$q = \bigvee_{h \in M^*}q_hh(m_h) = \bigvee_{h \in M^*}h(q_h m_h) = g((q_h m_h)_{h \in M^*}).$$

Therefore $g$ is a surjection.

Let us consider the image $g_*(1)$ of the unit of $Q$ under the residual map of $g$, and note that $g \circ g_* = \id_Q$ by the surjectivity of $g$. Let $k:q\in Q \mapsto q \cdot g_*(1) \in M^{M^*}$. For all $q\in Q$, $g(k(q))=q\cdot g(k(1))=q \cdot g(g_*(1)) = q \cdot 1=q$, i.e., $k$ is a right inverse for $g$ and, therefore, $Q$ is a retract of $M^{M^*}$.
\item[($\Leftarrow$)] Now suppose that $Q$ is a retract of $M^I$, for some set $I$, with \begin{tikzcd}
{M^I} \arrow[r, "\beta"', two heads, shift left=-1] & {Q} \arrow[l, "\alpha"', hook', shift right=1]
\end{tikzcd} being the retraction with its corresponding section. For all $i \in I$, let $\mu_i: M \to M^I$ be the canonical coproduct embedding, and $\beta_i=\beta \circ \mu_i \in M^*$. Then, for each $i \in I$, the following diagram commutes: 
$$\begin{tikzcd}
{M^{I}} \arrow[rrd, "\beta", two heads]                                  &  &    \\
                                                     &  & {Q} \\
{M} \arrow[rru, "\beta_i"'] \arrow[uu, "\mu_i", hook'] &  &   
\end{tikzcd}.$$
We have:
$$Q = \beta[M^I] = \beta\left[\la \bigcup_{i \in I} \mu_i[M]\ra\right] = \la\bigcup_{i \in I} \beta \mu_i[M]\ra = \la\bigcup_{i \in I} \beta_i[M]\ra \leq \operatorname{tr}(M).$$ 
Therefore $\operatorname{tr}(M) = Q$.
\end{enumerate}
\end{proof}

In the case of rings, the definitions of generator and separator are equivalent. For the case of quantales we have one of the implications, given by the following proposition. For integral quantales the converse implication holds, as we shall see in Corollary \ref{intsepgen}, while for the general case the question remains open.
\begin{proposition}\label{gensep}
Let $Q$ be a quantale and $\mm{}$ a $Q$-module. If $\mm{}$ is a generator for $\mc$, then it is a separator.
\end{proposition}
\begin{proof}
Since $M$ is a generator, by Proposition \ref{generchar}, $Q$ is a retract of $M^I$, for some set $I$. Let
\begin{tikzcd}
{M^I} \arrow[r, "\beta"', two heads, shift left=-1] & {Q} \arrow[l, "\alpha"', hook', shift right=1]
\end{tikzcd}
be the retraction with its corresponding section, and consider the canonical embedding $\alpha_i: M \hookrightarrow M^I$. Let $N$ and $P$ be left $Q$-modules, and let $f$ and $g$ be different morphisms from $N$ to $P$. Then there exists $y \in N$ such that $f(y)\neq g(y)$; let us define the $Q$-module morphism $h: q \in Q \mapsto qy \in N$. Note that $fh\neq gh: Q \rightarrow P$, since $fh(1)=f(y)\neq g(y)=gh(1)$. Hence, $fh\beta ((x_i)_{i \in I})=fh(1)=f(y)\neq g(y)=gh(1)=gh\beta ((x_i)_{i \in I})$ for every family $(x_i)_{i\in I} \in \beta^{-1}(1)$, and therefore $fh\beta \neq gh\beta: M^I\rightarrow P$.

On the other hand, for all $(x_i)_{i \in I} \in M^I$, $(x_i)_{i \in I} = \bigvee_{i \in I} \alpha_i(x_i)$, so
$$\bigvee_{i\in I}fh\beta \alpha_i(x_i)\neq \bigvee_{i\in I}gh\beta \alpha_i(x_i),$$
which implies that there exists $i \in I$ such that $fh\beta \alpha_i(x_i)\neq gh\beta \alpha_i(x_i)$. Then, $fh\beta \alpha_i \neq gh\beta \alpha_i: M \rightarrow P$.
\end{proof}

\begin{lemma}\label{lema1} 
Let $M$ be a $Q$-module and let $\top '=\bigvee \operatorname{tr}(M)$. If $M$ is a separator, then $\top' = \top$.
\end{lemma}
\begin{proof} 
We shall prove that $\top' < \top$ implies that $M$ is not a separator. Consider the interval $I=[\bot,\top '] \subsetneq Q$, since $\top \notin I$, $I$ is a proper left $Q$-ideal of $Q$. Thus, by \cite[Theorem 3.8]{russajl}, there exists a $Q$-module congruence $\vartheta$ such that $I=\bot/\vartheta$.
Since $I\neq Q$, $Q/\vartheta \neq \{\bot/\vartheta\}$, hence we can consider the morphisms
$$\pi, \overline{0}:Q\to Q/\vartheta,$$
where $\pi$ is the canonical projection and $\overline{0}$ is the $\bot$-constant morphism, which are clearly distinct. For all $h:M\to Q$ and $m\in M$, we have $h(m)\leq \top'$,
which implies $\pi(h(m))\leq \pi(\top ')=\bot/\vartheta = \overline{0}(h(m))$. So, $M$ does not separate $\pi$ and $\overline{0}$, and therefore $M$ is not a separator.
\end{proof}

\begin{corollary}\label{intsepgen}
If $Q$ is an integral quantale, then a $Q$-module $M$ is a generator for $\mc$ if and only if $M$ is a separator.
\end{corollary}
\begin{proof} 
One implication holds for all quantales, by Proposition \ref{gensep}. For the converse implication, suppose that $M$ is not a generator, i.e., $\operatorname{tr}(M) \neq Q$. Since $\top'=\bigvee\tr(M)\leq\top=1$, and $1 \notin \tr(M)$, we have that $\bigvee\tr(M)<1$, i.e., the top element of $\tr(M)$ is not the greatest element of $Q$. Then Lemma \ref{lema1} guarantees that $M$ is not a separator. 
\end{proof}

\begin{proposition}\label{alphaiso}
A left $Q$-module $\mm{}$ is a generator for $\mc$ iff the $(Q,Q)$-homomorphism
$$\alpha: \bigvee_{i \in I}x_i\tensor f_i \in M \tensor_E M^* \mapsto \bigvee_{i \in I}f_i(x_i) \in Q$$
is a surjection. Moreover, if $\alpha$ is a surjection, then it is an isomorphism.
\end{proposition}
\begin{proof}
The first statement follows from the definition of $\alpha$ and the Proposition \ref{generchar}.  
For the second statement, suppose that $\alpha$ is a surjection, and let us show that it is an injection too. So, let 
$$\bigvee_{i \in I}f_i(x_i) = \alpha \left(\bigvee_{i \in I}x_i\tensor f_i\right) = \alpha \left(\bigvee_{j \in j}x_j'\tensor f_j'\right) = \bigvee_{j \in J}f_j'(x_j'),$$
and
$$\alpha \left(\bigvee_{k \in K}x_k''\tensor f_k''\right) = \bigvee_{k \in K}f_k(x_k) = 1,$$
for some sets $I$, $J$ and $K$. 

So, using the $MNM$-associativity and the $NMN$-associativity, we have
$$\begin{array}{ll}
\bigvee\limits_{i \in I}x_i\tensor f_i & = \bigvee\limits_{i \in I}\left(\bigvee\limits_{k \in K}(x''_kf''_k)x_i\right)\tensor f_i \\ & = \bigvee\limits_{k \in K}\left(\bigvee\limits_{i \in I}(x''_kf''_k)x_i\tensor f_i\right) \\ & = \bigvee\limits_{k \in K}\left(\bigvee\limits_{i \in I}x''_k(f''_kx_i)\tensor f_i\right) \\ & = \bigvee\limits_{k \in K}\left(\bigvee\limits_{i \in I}x''_k\tensor (f''_kx_i)f_i\right) \\ & = \bigvee\limits_{k \in K}\left(\bigvee\limits_{i \in I}x''_k\tensor f''_k\cdot(x_if_i)\right) \\ & = \bigvee\limits_{k \in K}\left(x''_k\tensor f''_k\bigvee\limits_{i \in I}x_if_i\right) \\ & = \bigvee\limits_{k \in K}\left(x''_k\tensor f''_k\bigvee\limits_{j \in J}x'_jf'_j\right) \\ & = \bigvee\limits_{k \in K}\left(\bigvee\limits_{j \in J}x''_k\tensor f''_k\cdot(x'_jf'_j)\right) \\ & =\bigvee\limits_{k \in K}\left(\bigvee\limits_{j \in J}x''_k\tensor (f''_kx'_j)f'_j\right) \\ & = \bigvee\limits_{k \in K}\left(\bigvee\limits_{j \in J}x''_k(f''_kx'_j)\tensor f'_j)\right) \\ & = \bigvee\limits_{k \in K}\left(\bigvee\limits_{j \in J}(x''_kf''_k)x'_j\tensor f'_j\right) \\ & = \bigvee\limits_{j \in J}\left(\bigvee\limits_{k \in K}\left((x''_kf''_k)x'_j\right)\tensor f'_j\right) \\ & = \bigvee\limits_{j \in J}x'_j\tensor f'_j.
\end{array}$$
\end{proof}

\begin{definition}
A $Q$-module $\QM$ is said to be a \emph{progenerator} for the category $\mc$ if it is a projective generator.
\end{definition}

From Propositions \ref{betaiso} and \ref{alphaiso}, we get the following important result:

\begin{theorem}\label{progiso}
A $Q$-module $\QM$ is a progenerator iff the homomorphisms $\alpha: M \tensor_E M^* \rightarrow Q$ and $\beta: M^* \tensor_Q M \rightarrow E$ are bimodule isomorphisms. 
\end{theorem}

\begin{proposition}\label{prog}
Let $\mm{}$ be a progenerator for $\mc$. Then:
\begin{enumerate}[(i)]
\item $N = M^* := \mc(\mm{},\mm[Q]{}) \cong \MC[E](M_E,E_E)$ as $E$-$Q$-bimodules.
\item $M \cong \mc[E]({}_{E}N,{}_{E}E)$ as $Q$-$E$-bimodules.
\item $Q \cong \End(M_E) \cong \End({}_{E}N)$ as quantales.
\item $M \cong \MC(N_Q,Q_Q)$ as $Q$-$E$-modules.
\item $E \cong \End(N_Q)$ as quantales.
\end{enumerate}   
\end{proposition}
\begin{proof}
\begin{enumerate}[(i)]
\item Let, for each $f \in M^*$, $\lambda(f) \in \MC[E](M_E,E_E)$ be the map defined by $\lambda(f)(m) = fm \in E$ for all $m \in M$. Then we get the map $\lambda: M^* \rightarrow \MC[E](M_E,E_E)$, and we shall prove that it is an $E$-$Q$-module isomorphism.

First, we note that NME-associativity implies that $\lambda(f) \in \MC[E](M_E,E_E)$, while ENM-associativity and NMQ-associativity guarantee that $\lambda$ is an $E$-$Q$-bimodule homomorphism.

Since $M$ is a generator for $\mc$, like in Proposition \ref{alphaiso}, there exists a set $K$, $\{x_k\}_{k \in K} \subseteq M$, and $\{f_k\}_{k \in K} \subseteq N$ such that $\bigvee_{k \in K}x_kf_k = \id_Q$. So, suppose $\lambda(f)(m) = fm = f'm = \lambda(f')(m)$ for some $f, f' \in Q$ and all $m \in M$. Then, using the NMN-associativity, we have 
$$\begin{array}{lllll} f &= f\id_Q & = f\bigvee\limits_{k \in K}x_kf_k & = \bigvee\limits_{k \in K}f\cdot (x_kf_k) & = \bigvee\limits_{k \in K}(fx_k)f_k \\ & = \bigvee\limits_{k \in K}(f'x_k)f_k & = \bigvee\limits_{k \in K}f'\cdot(x_kf_k) & = f'\bigvee\limits_{k \in K}x_kf_k & = f',
\end{array}$$
which proves that $\lambda$ is injective. In order to show that it is also surjective, let us consider a homomorphism $f \in \MC[E](M_E,E_E)$ and an $m \in M$. Using the MNM-associativity and the ENM-associativity, we have
$$\begin{array}{lll}
f(m) &= f\left(\left(\bigvee\limits_{k \in K}x_kf_k\right)m\right) &= f\left(\bigvee\limits_{k \in K}(x_kf_k)m\right) \\ & = f\left(\bigvee\limits_{k \in K}x_k(f_km)\right) & = \bigvee\limits_{k \in K}f(x_k)\cdot(f_km) \\ & = \bigvee\limits_{k \in K}(f(x_k)f_k)m & = \lambda\left(\bigvee\limits_{k \in K}(f(x_k)f_k)\right)(m).
\end{array}$$

\item It can be proved in a similar way to the previous one.
\item Consider the quantale homomorphisms $\sigma: Q \rightarrow \End(M_E)$ and $\tau: Q \rightarrow \End({}_{E}N)$, defined by $\sigma(q)(m) = qm$ and $\tau(q)(f) = fq$ for any $q \in Q$, $m \in M$ and $f \in N$. In the same way as in (i), we show that $\sigma$ is injective. Indeed, using the QMN-associativity, suppose that $qm = \sigma(q)(m) = \sigma(q')(m) = q'm$, for some $q, q' \in Q$, $m \in M$. Then
$$\begin{array}{lllll}
q & = q\id_Q & = q\bigvee\limits_{k \in K}x_kf_k & = \bigvee\limits_{k \in K}q\cdot(x_kf_k) & = \bigvee\limits_{k \in K}(qx_k)f_k \\
& = \bigvee\limits_{k \in K}(q'x_k)f_k & = \bigvee\limits_{k \in K}q'(x_kf_k) & = q'\bigvee\limits_{k \in K}x_kf_k & = q'.\end{array}$$ 
Furthermore, just like in (i), we can show that $\sigma\left(\bigvee_{k \in K}f(x_k)f_k\right) = f$ for any $f \in \End(M_E)$, so $\sigma$ is also surjective and, therefore, an isomorphism. The proof that $\tau$ is also an isomorphism is completely analogous.
\end{enumerate}
The proofs of (iv) and (v) are analogous to the ones of (i) and (iii) respectively, using the fact that, by Proposition \ref{betaiso}, there exist a finite set $K$, $\{x_k\}_{k \in K} \subseteq M$, and $\{f_k\}_{k \in K} \subseteq N$, such that $\bigvee_{k \in K}f_kx_k = \id_R$.
\end{proof}

\begin{corollary}\label{corprog}
Let $\mm{}$ be a progenerator for $\mc$. Then $M_E$, $N_Q$ and ${}_{E}N$ are also progenerators for $\MC[E]$, $\MC[Q]$ and $\mc[E]$, respectively, and $\alpha$ and $\beta$ are isomorphisms.
\end{corollary}
\begin{proof}
For ${}_{E}N$, by (ii) and (iii) of the previous proposition, we have $M \cong \mc[E]({}_{E}N,{}_{E}E)$ and $Q \cong \End({}_{E}N)$. Therefore, considering that $\alpha$ and $\beta$ are surjections and applying Propositions \ref{betaiso} and \ref{alphaiso} to the module ${}_{E}N$, it follows that ${}_{E}N$ is a progenerator for $\mc[E]$. The proofs for the other cases are analogous.
\end{proof}

\section{Morita Equivalence}
\label{morita}

To present the concept of Morita equivalence we need to recall some general categorical notions in the setting of module categories over quantales. A pair of homomorphisms $\alpha,\beta: A \rightarrow B$ in a category of right Q-modules $\MC$ is called a \textit{kernel pair} of a homomorphism $\gamma : B \rightarrow C$ in $\MC$, $(\alpha,\beta) = \operatorname{kerp}(\gamma)$, if the commutative square 

\begin{equation}\label{absinterdiag}
\xymatrix{
A \ar[r]^{\alpha} \ar[d]_{\beta}  & B \ar[d]^\gamma
&\\B \ar[r]_{\gamma}  & C
\\
}
\end{equation}
is a pullback. An \textit{equalizer} $(E,\mu)$ of the kernel pair $(\alpha,\beta)$ is a morphism $\mu : E \rightarrow A$ such that (i) $\alpha \mu = \beta \mu$, and (ii) for every homomorphism $\nu : Y \rightarrow A$ with $\alpha\nu = \beta\nu$, there exists exactly one homomorphism $\omega : Y \rightarrow E$ such that $\nu=\mu\omega$.

A homomorphism $\gamma : B \rightarrow C$ is called a \textit{coequalizer} of a pair of morphisms $\alpha,\beta: A \rightarrow B$ in $\MC$, $\gamma = \operatorname{coeq}(\alpha,\beta)$, if $\gamma\alpha = \gamma\beta$, and for every homomorphism $\eta : B \rightarrow D$ with $\eta\alpha = \eta\beta$, there exists exactly one homomorphism $\delta : C \rightarrow D$ such that $\eta=\delta\gamma$. A diagram $\begin{tikzcd}
{A}  \arrow[r, "\alpha", shift right=-1] \arrow[r, "\beta"', shift right=1] & {B} \stackrel{\gamma}{\longrightarrow} C
\end{tikzcd}$ is called \textit{left exact} (respectively, \textit{right exact}) if $(\alpha,\beta) = \operatorname{kerp}(\gamma)$ (resp., $\gamma = \operatorname{coeq}(\alpha,\beta)$), and \textit{exact} if it is left and right exact simultaneously. A functor $F : \MC \rightarrow \MC[R]$ between the module categories $\MC$ and $\MC[R]$ is said to be \textit{left continuous} (\textit{exact}), \textit{right continuous} (\textit{exact}), and \textit{continuous} (\textit{exact}) if it preserves left exact, right exact, and exact diagrams, respectively.


The following two characterizations are, respectively, 7.4.2 and 8.4.2 of \cite{schu}.
 
\begin{proposition}\label{comp}\label{cocomp}
A category $\cat C$ is complete if and only if it possesses equalizers and products. It is cocomplete if and only if it has coequalizers and coproducts.
\end{proposition}

\begin{proposition}\label{coecocomp}
The category $\MC{}$ is complete and cocomplete.
\end{proposition}
\begin{proof}
By \cite[Proposition 4.13]{rusjlc}, $\MC{}$ has products and coproducts, hence, by Proposition \ref{comp}, it is enough show that the category it also has equalizers and coequalizers.

So, let $f$ and $g$ be morphisms from $M$ to another module $N$ and let $A = \{x \in M: f(x)=g(x)\}$, which is easily seen to be a submodule of $M$. The diagram

\begin{tikzcd}
A \arrow[r, "e"] \arrow[r, "\subseteq"'] & M \arrow[r, phantom] \arrow[r, "g"', shift right] \arrow[r, "f", shift left] & N \\
                                           & A' \arrow[u, "e'"'] \arrow[lu, "e''", dashed]                                      &  
\end{tikzcd}
shows that the inclusion morphism $e$ is equalizer of $(f,g)$. Indeed, by the definition of $A$, we have $f e = g e$. If $e':A' \rightarrow M$ is a homomorphism such that $f e' = g e'$, then the image of $e'$ is contained in $A$, hence the same morphism $e'': y \in A' \mapsto e'(y) \in A$ is the unique homomorphism such that $e' = e e''$.

For the coequalizer, in the diagram
\begin{tikzcd}
                                                          &                                   & B'                        \\
M \arrow[r, "g"', shift right] \arrow[r, "f", shift left] & N \arrow[r, "c"] \arrow[ru, "c'"] & B \arrow[u, "d"', dashed]
\end{tikzcd}
 
let us consider the congruence $\vartheta = \langle(f(x),g(x)): x \in M\rangle$, and let $B=N/\vartheta$ and $c$ be the canonical projection. Given any homomorphism $c':N \rightarrow B'$ such that $c' f = c' g$, we get $\vartheta = \ker c \subseteq \ker c'$, hence there exists a unique homomorphism $d:B \rightarrow B'$ such that $c' = d c$. 
\end{proof}

\begin{theorem}\label{adj}
For a functor $F : \MC \rightarrow \MC[R]$ the following statements are equivalent:
\begin{enumerate}[(a)]
\item $F$ has a right adjoint;
\item $F$ is right continuous and preserves coproducts;
\item There exists (unique up to natural isomorphism) a $Q$–$R$-bimodule $M \in \mc_R$ such that the functors $-\tensor_Q M : \MC \rightarrow \MC[R]$ and $F$ are naturally isomorphic, i.e., $F \cong -\tensor_Q M$.
\end{enumerate}   
\end{theorem}
\begin{proof}
The implication $(a) \Rightarrow (b)$ follows from \cite[V.5 Theorem 1]{maclane}, and $(c) \Rightarrow (a)$ follows from Proposition \ref{natiso}.

For $(b) \Rightarrow (c)$, let $P := F(Q) \in \MC[R]{}$, then the functor $F$ induces a quantale homomorphism $Q \cong \MC{}(Q,Q) \rightarrow \MC[R]{}(F(Q),F(Q))=\MC[R]{}(P,P)$, which turns $P$ a $Q$-$R$-bimodule, i.e., $P \in \mc{}_R$. Let $M \in \mc{}$ be an arbitrary module  and a surjection $\gamma: Q^I \twoheadrightarrow M$ for some free module $Q^I \in \MC{}$. 
Since any surjective homomorphism in $\MC$ is a coequalizer of some pair of homomorphisms, let $\alpha,\beta:A \rightarrow Q^I$ be such that $\gamma=\operatorname{coeq}(\alpha,\beta)$, and therefore, for some surjection $\theta:Q^J \twoheadrightarrow A$ there is a right exact diagram
$$\begin{tikzcd}
{Q^J}  \arrow[r, "\alpha\theta", shift right=-1] \arrow[r, "\beta\theta"', shift right=1] & {Q^I} \stackrel{\gamma}{\twoheadrightarrow} M.
\end{tikzcd}$$
Applying the functors $-\tensor_Q P$ and $F$ to this diagram, we obtain the following commutative diagram in $\MC[R]{}$
$$\begin{tikzcd}
{P^J}\arrow[transform canvas={xshift=0.3ex},-]{d} \arrow[transform canvas={xshift=-0.4ex},-]{d} \arrow[r, shift right=-1] \arrow[r, shift right=1] & {P^I} \arrow[transform canvas={xshift=0.3ex},-]{d} \arrow[transform canvas={xshift=-0.4ex},-]{d}\arrow[r, shift right=0] &  M\tensor_Q P \\
{(F(Q))^J} \arrow[r, shift right=-1] \arrow[r, shift right=1]  & (F(Q))^I \arrow[r, shift right=0] &  {F(M)}\\
\end{tikzcd},$$
in which both rows are right exact diagrams. Then there exists an isomorphism between $M\tensor_Q P$ and $F(M)$ which completes the diagram above, and it can easily be verified to be a natural one in $M \in \mc$. Hence, $F \cong -\tensor_Q P$.
\end{proof}

\begin{lemma}
An epimorphism $\theta$ in $\MC$ is a surjection iff $\theta = \operatorname{coeq}(\operatorname{kerp}(\theta))$.
\end{lemma}
\begin{proof}
This follows from the observations that in $\MC{}$ any equalizer is a surjection and any surjection is coequalizer of some pair of homomorphisms, and from \cite[Proposition 2.5.7]{borc}.
\end{proof}

\begin{definition}
Two quantales $Q$ and $R$ are said to be \emph{Morita equivalent} (in symbols: $Q\approx R$) iff there is an equivalence of the category of (left) modules over $Q$, $\mc$, and the category of (left) modules over $R$, $\mc[R]$. It can be shown that the left module categories $\mc$ and $\mc[R]$ are equivalent if and only if the right module categories $\MC$ and $\MC[R]$ are equivalent.
\end{definition}


\begin{proposition}
Let \begin{tikzcd}
{F:\MC}  \arrow[r, shift left=1] & {\MC[R]:G} \arrow[l, shift right=-1]  \end{tikzcd} be a Morita equivalence between $Q$ and $R$, i.e., $Q\approx R$. The following hold:
\begin{enumerate}[(i)]
\item If $\theta$ a surjection in $\MC$, then $F(\theta)$ is a surjection in $\MC[R]$
\item If $M$ is a generator for $\MC[Q]$, then $F(M)$ is a generator for $\MC[R]$.
\item If $P \in \MC[Q]$ is projective, then $F(P) \in \MC[R]$ is projective too.
\end{enumerate} 
\end{proposition}
\begin{proof}
By \cite[IV.4, page 93]{maclane}, $F \dashv G$ and $G \dashv F$, i.e. the functor $F$ is both left and right adjoint to $G$. Then, by of \cite[V.5, Theorem 1]{maclane} and its dual, $F$ preserves limits and colimits, and hence the (i) follows from the previous lemma.

Items (ii) and (iii) are obvious because being a generator and projectivity are categorical properties, hence they are preserved by equivalences.
\end{proof}

\begin{corollary}\label{pres}
A Morita equivalence between two quantales preserves projective generators.
\end{corollary}

\begin{lemma}\label{isoquan}
Let $R$ be a $R$-$R$-bimodule and $M$ a $Q$-$R$-bimodule such that $R$ and $\End(\mm{})$ are isomorphic as $R$-$R$-bimodules. Then $R$ and $\End(\mm{})$ are isomorphic as quantales.
\end{lemma}
\begin{proof}
Let $\alpha : R \to \End(\mm{})$ be an $R$-$R$-bimodule isomorphism. There exist $r_1 \in R$ and $f_1 \in \End(\mm{})$ such that $\alpha(r_1) = \id_M$ and $\alpha(1_R) = f_1$, and they are uniquely determined. 

For all $r \in \End(\mm{})$, let us define $\hat r \in \End(\mm{})$ by $\hat r: x \mapsto x \cdot r$, and observe that, for all $x \in {}_{Q}M$,
\begin{equation*}\label{prim}
(f_1 \circ \hat r)(x) = f_1(\hat r(x)) = f_1(x \cdot r) = f_1(x) \cdot r = \hat r(f(x_1)) = (\hat r \circ f_1)(x).
\end{equation*}
Moreover, for all $r \in R$, we have
\begin{equation}\label{seg}
\alpha(r) = \alpha(1_R \cdot r) = \alpha(1_R) \cdot r = f_1 \cdot r = f_1 \circ \hat r = \hat r \circ f_1;
\end{equation}
in particular, for $r = r_1$,
\begin{equation}\label{terc}
\id_M = \alpha(r_1) = \hat r_1 \circ f_1.
\end{equation}

Now let $h : r \in R \mapsto \hat r \in \End(\mm{})$. $h$ is clearly a quantale homomorphism. Let us prove that $h$ is onto. For all $g \in \End(\mm{})$, the composition $f_1 \circ g$ also belongs to $\End(\mm{})$. Since $\alpha$ is bijective, there exists an $r \in R$ such that $\alpha(r) = f_1 \circ g$. So we have,
$$\begin{array}{lll} & \alpha(r) = f_1 \circ g & \\
\To & f_1 \circ \hat r = f_1 \circ g & \mbox{ (by (\ref{seg})) } \\
\To & \hat r_1 \circ f_1 \circ \hat r = \hat r_1 \circ f_1 \circ g \\
\To & \hat r = g & \mbox{ (by (\ref{terc}))} \\
\To & h(r) = g. & \end{array}$$
Therefore, for all $g \in \End(\mm{})$, there exists an $r \in R$ such that $h(r) = g$, i.e., $h$ is surjective.

In order to prove injectivity, let $r, s \in R$ be such that $h(r) = h(s)$. Then, by (\ref{seg}), we have $\alpha(r) = \hat{r} f_1 = h(r) f_1 = h(s) f_1 = \hat{s} f_1 = \alpha(s)$, but $\alpha$ is bijective, and therefore such equalities imply $r = s$.
\end{proof}

\begin{proposition}
Let $\alpha' : P \tensor_R G \rightarrow Q$ and $\beta' : G \tensor_Q P \rightarrow R$ be a $Q$-$Q$-isomorphism and an $R$-$R$-isomorphism respectively, for bimodules $P \in \mc_R$ and $G \in \mc[R]_Q$ over quantales $Q$ and $R$. Then the left module $\mm[P]{}$ is a progenerator for the category of left modules $\mc$, and there are quantale isomorphisms $R \cong \End(\mm[P]{})$ and $Q \cong \End(P_R)$, as well as an $R$-$Q$-isomorphism $G \cong P^* = \mc(\mm[P]{},\mm[Q]{})$ between the $R$-$Q$-bimodules $G$ and $P^*$.
\end{proposition}
\begin{proof}
Applying 
$$\mc[Q]_Q(P \tensor_R G, Q) \cong \mc[R]_Q(G,P^*)$$ from Proposition \ref{natiso}, we have the $R$-$Q$-module homomorphism $\gamma : {}_RG_Q \to {}_RP^*_Q$, 
defined by $p(\gamma(g)) = \alpha'(p \otimes g).$ 

Then, applying Proposition \ref{alphaiso} to the commutative diagram

$$
\begin{tikzcd}[column sep=large]
P \otimes_R G \arrow[d, "\alpha'"'] \arrow[r, "1_P \otimes \gamma"] & P \otimes_R P^* \arrow[d, "\alpha"] \\
Q \arrow[r, phantom, "\scriptstyle ="] & Q
\end{tikzcd}
$$

we obtain that $\mm[P]{}$ is a generator for $\mc$ and $\alpha$ is an isomorphism. Furthermore, applying the functor $G \otimes_Q -$ to this diagram, using Proposition \ref{natisom} and the fact that $\beta' : G \otimes_Q P \to R$ is an isomorphism, we get that $\gamma$ is also an isomorphism. Applying again the functor $G \otimes_Q -$ to the commutative diagram
$$
\begin{tikzcd}
(P \otimes_R P^*) \otimes_Q P \arrow[r, "1_P \otimes \beta"] \arrow[d, "\alpha \otimes 1_P"'] & P \otimes_R R \arrow[d, "\cong"] \\
Q \otimes_Q P \arrow[r, "\scalebox{0.9}{$\cong$}"', shorten <=14pt, shorten >=14pt] & P
\end{tikzcd},
$$
by Proposition \ref{natisom} and the fact that $\beta' : G \otimes_Q P \to R$ is an isomorphism, we get that $\beta : P^* \otimes_Q P \to R$ is also an isomorphism. Hence, by Theorem \ref{progiso}, $\mm[P]{}$ is a progenerator for $\mc$.

Now, applying the functor $P \otimes_R -$ to the $R$-$R$-module $\End(G \tensor_Q P)$, combining Proposition \ref{natisom}, the isomorphism $
\beta' : G \otimes_Q P \to R$, the functor $P \otimes_R (G \otimes_Q -)$ to the $R$-$R$-module $\End(\mm[P]{})$, and the isomorphism $\alpha' : P \otimes_R G \to Q$, one readily sees that the $R$-$R$-modules $(G \otimes_Q P)$ and $\End(\mm[P]{})$ are isomorphic.

Then, applying Lemma \ref{isoquan} and the $R$-$R$-isomorphism $\beta': G \otimes_Q P \to R$, we have that $R \cong \End(\mm[P]{})$ as quantales. Hence, identifying the quantales $R$ and $\End(\mm[P]{})$, using also Proposition \ref{prog}, we conclude that $Q \cong \End(P_R)$ as quantales.
\end{proof}

\begin{corollary}
Let $Q, R$, and $S$ be quantales, $\mm[P]{}$ be a progenerator for $\mc{}$ such that the quantales $R$ and $\operatorname{End}(\mm[P]{})$ are isomorphic, and ${}_{R}G$ a progenerator of $\mc[R]{}$ such that the quantales $S$ and $\operatorname{End}({}_{R}G)$ are isomorphic. Then there exists a left progenerator $\mm[(P\tensor_R G)]{}$ for $\mc{}$ such that the quantales $S$ and $\operatorname{End}(\mm[(P\tensor_R G)]{})$ are isomorphic, and  the $S$-$Q$-modules $(P\tensor_R G)^*$ and $(G^*\tensor_R P^*)$ are isomorphic.
\end{corollary}
\begin{proof}
By Theorem \ref{progiso}, there are the bimodule isomorphisms

$$\alpha_P: P \tensor_R P^* \rightarrow Q \hspace{7mm}\text{ and } \hspace{7mm} \beta_P: P^* \tensor_Q P \rightarrow R$$
and
$$\alpha_M: M \tensor_S M^* \rightarrow R \hspace{7mm} \text{ and } \hspace{7mm} \beta_M: M^* \tensor_R M \rightarrow S.$$

Using such isomorphisms and Proposition \ref{natisom}, we obtain the bimodule isomorphisms 

$$(P \tensor_R M) \tensor_S (M^* \tensor_R P^*) \cong Q \hspace{7mm} \text{ and } \hspace{7mm} (M^* \tensor_R P^*) \tensor_S (P \tensor_R M) \cong S.$$
Therefore, the result follows from the previous proposition.
\end{proof}

\begin{theorem}\label{moritath}
Two quantales $Q$ and $R$ are Morita equivalent if and only if 
there exists a progenerator $\mm[P]{}$ for $\mc{}$ such that the quantales $R$ and $\operatorname{End}(\mm[P]{})$ are isomorphic.
\end{theorem}
\begin{proof}
For the left-to-right implication, suppose that $Q$ and $R$ are Morita equivalent. By Theorem \ref{adj} we have that $F\cong -\tensor_Q P$ for the $Q$-$R$-bimodule $P=F(Q) \in \MC[R]{}$, and $G\cong -\tensor_R T$ for the $R$-$Q$-bimodule $T=G(R) \in \MC{}$. First notice that by Proposition \ref{pres} the modules $P=F(Q) \in \MC[R]{}$ and $T=G(R) \in \MC{}$ are progenerators for $\MC[R]{}$ and $\MC{}$ respectively. Then, taking into consideration the natural isomorphisms, we have:
$$\begin{array}{llll}
\MC[R]{}(P,P) & = \MC[R]{}(F(Q),P) & \cong \MC{}(Q,G(P)) & \\ & \cong \MC{}(Q,G(F(Q))) & \cong \MC{}(Q,Q) &\cong Q,
\end{array}$$
i.e., $\operatorname{End}(P_{R}) \cong Q $, that is equivalent to desired conclusion.

In order to prove the right-to-left implication, let $\mm[P]{}$ be a progenerator for $\mc{}$ such that the quantales $R$ and $\operatorname{End}(\mm[P]{})$ are isomorphic, and $P \in \mc_R$,  and let $N = P^* := \mc{}(\mm[P]{}, \mm[Q]{})$ be the dual $R$-$Q$-module $P^*$ of the $Q$-$R$-module P. Considering the composition of the functors $-\tensor_Q P: \MC{} \rightarrow \MC[R]{}$ and $-\tensor_R N = -\tensor_R P^*: \MC[R]{} \rightarrow \MC{}$ and using Proposition \ref{natiso} and Theorem \ref{progiso}, we have the commutative diagram in $\MC{}$:
$$\begin{tikzcd}
X \tensor_Q P \tensor_R P^* \arrow[d, "\gamma \tensor 1_P \tensor 1_{P^*}"'] & {} \arrow[r, "1_X\tensor \alpha"] & {} & X \tensor_Q Q \arrow[d, "\gamma \tensor 1_Q"] & {} \arrow[r, "\cong"] & {} & X \arrow[d, "\gamma"'] \\
Y \tensor_Q P \tensor_R P^*                                                  & {} \arrow[r, "1_Y\tensor \alpha"] & {} & Y \tensor_Q Q                                 & {} \arrow[r, "\cong"] & {} & Y                     
\end{tikzcd}$$
for any $\gamma: X \rightarrow Y \in \MC{}$. From this diagram one can easily see that the functors $-\tensor_R P^* \circ -\tensor_Q P: \MC{} \rightarrow \MC{}$ and $\operatorname{Id}_{\MC{}}: \MC{} \rightarrow \MC{}$ are naturally isomorphic, i.e., $-\tensor_R P^* \circ -\tensor_Q P \cong \operatorname{Id}_{\MC{}}$. Similarly, by Proposition \ref{prog} and Corollary \ref{corprog}, one obtains the functor natural isomorphism $-\tensor_Q P$ $\circ$ $-\tensor_R P^* \cong \operatorname{Id}_{\MC[R]{}}$, and the equivalence of categories $-\tensor_Q P: \MC{} \rightleftarrows  \MC[R]{}: -\tensor_R P^*$. 
\end{proof}

\section*{Declarations}

\begin{itemize}
\item[] {\bf Ethical Approval}

This declaration is not applicable.

\item[] {\bf Competing interests}

Neither the authors nor third parties have financial interests related to this work.

\item[] {\bf Authors' contributions}

The authors contributed equally to this work.

\item[] {\bf Funding}

No funding received.

\item[] {\bf Availability of data and materials}

This declaration is not applicable.
\end{itemize}

\newpage

\bibliographystyle{plain}

\bibliography{cirobibtex}

@article{pas2001,
  author    = {Paseka, Jan},
  title     = {Morita contexts and their lattices of relations},
  journal   = {Acta Mathematica Universitatis Comenianae},
  year      = {2001},
  volume    = {70},
  number    = {2},
  pages     = {155--176}
}

@article{luh,
  author    = {Luha{\"a}{\"a}r, Urmas},
  title     = {Enlargements of Quantales},
  journal   = {Mathematica Slovaca},
  year      = {2023},
  volume    = {73},
  number    = {5},
  pages     = {1119--1132},
  doi       = {10.1515/ms-2023-0082}
}

@inproceedings{pas2002,
  author    = {Paseka, Jan},
  title     = {Morita equivalence in the context of Hilbert modules},
  booktitle = {Proceedings of the Ninth Prague Topological Symposium (Prague, 2001)},
  publisher = {Topology Atlas},
  address   = {Toronto},
  year      = {2002},
  pages     = {223--251},
  note      = {arXiv:math/0204112}
}

@article{pas2005a,
  author    = {Paseka, Jan},
  title     = {Characterization of Morita Equivalence Pairs of Quantales},
  journal   = {International Journal of Theoretical Physics},
  year      = {2005},
  volume    = {44},
  number    = {7},
  pages     = {875--883},
  doi       = {10.1007/s10773-005-7065-8}
}

@inproceedings{pas2005b,
  author    = {Paseka, Jan},
  title     = {Morita equivalence for $m$-regular quantales},
  booktitle = {Contributions to General Algebra 16: Proceedings of the Dresden Conference 2004 (AAA68) and the Summer School 2004},
  publisher = {Verlag Johannes Heyn},
  address   = {Klagenfurt},
  year      = {2005},
  pages     = {155--171},
  isbn      = {3-7084-0163-8}
}

@incollection{pas99,
 author = {Paseka, J.},
 title = {Hilbert {{\(Q\)}}-modules and nuclear ideals in the category of {{\(\bigvee\)}}-semilattices with a duality},
 booktitle = {CTCS '99. Conference on category theory and computer science, Edinburgh, GB, September 10--12, 1999},
 pages = {19},
 year = {1999},
 publisher = {Amsterdam: Elsevier},
 }

@article{nief96,
 author = {Niefield, Susan B.},
 title = {Constructing quantales and their modules from monoidal categories},
 fjournal = {Cahiers de Topologie et G{\'e}om{\'e}trie Diff{\'e}rentielle Cat{\'e}goriques},
 journal = {Cah. Topologie G{\'e}om. Diff{\'e}r. Cat{\'e}goriques},
 issn = {0008-0004},
 volume = {37},
 number = {2},
 pages = {163--176},
 year = {1996},
 }

@article{rose94,
 author = {Rosenthal, Kimmo I.},
 title = {Modules over a quantale and models for the operator ! in linear logic},
 fjournal = {Cahiers de Topologie et G{\'e}om{\'e}trie Diff{\'e}rentielle Cat{\'e}goriques},
 journal = {Cah. Topologie G{\'e}om. Diff{\'e}r. Cat{\'e}goriques},
 issn = {0008-0004},
 volume = {35},
 number = {4},
 pages = {329--333},
 year = {1994},
 }

@article{mesa2025,
 author = {Mesablishvili, Bachuki},
 title = {A characterization of module categories over quantales},
 fjournal = {Transactions of A. Razmadze Mathematical Institute},
 journal = {Trans. A. Razmadze Math. Inst.},
 issn = {2346-8092},
 volume = {179},
 number = {1},
 pages = {173--175},
 year = {2025},
 }

@article{moreq4,
 author = {Ad{\'a}mek, Ji{\v{r}}{\'{\i}} and Sobral, Manuela and Sousa, Lurdes},
 title = {Morita equivalence of many-sorted algebraic theories},
 fjournal = {Journal of Algebra},
 journal = {J. Algebra},
 issn = {0021-8693},
 volume = {297},
 number = {2},
 pages = {361--371},
 year = {2006},
}

@article{moreq3,
 author = {Lindner, Harald},
 title = {Morita equivalences of enriched categories},
 fjournal = {Cahiers de Topologie et G{\'e}om{\'e}trie Diff{\'e}rentielle Cat{\'e}goriques},
 journal = {Cah. Topologie G{\'e}om. Diff{\'e}r. Cat{\'e}goriques},
 issn = {0008-0004},
 volume = {15},
 pages = {377--397},
 year = {1974},
 }

@article{moreq2,
 author = {Marcus, Andrei},
 title = {Equivalences induced by graded bimodules},
 fjournal = {Communications in Algebra},
 journal = {Commun. Algebra},
 issn = {0092-7872},
 volume = {26},
 number = {3},
 pages = {713--731},
 year = {1998},
 }

@article{morita,
 author = {Morita, Kiiti},
 title = {Duality for modules and its applications to the theory of rings with minimum condition},
 fjournal = {Science Reports of the Tokyo Kyoiku Daigaku, Section A},
 journal = {Sci. Rep. Tokyo Kyoiku Daigaku, Sect. A},
 issn = {0371-3539},
 volume = {6},
 pages = {83--142},
 year = {1958},
 }

@article{eil60,
author = {Eilenberg, S.},
journal = {J. Indian Math. Soc.},
number = {24},
pages = {231-234},
title = {Abstract description of some basic functors},
year = {1960},
}

@article{wat60,
author = {Watts, C. E.},
journal = {Proc. Am. Math. Soc.},
number = {11},
pages = {5-8},
title = {Intrinsic characterizations of some additive functors},
year = {1960},
}

@book{schu,
author = {Schubert, H.},
publisher = {Springer-Verlag},
title = {Categories},
year = {1972},
}

@book{borc,
author = {Borceux, F.},
publisher = {Cambridge University Press},
series = {Encyclopedia of {M}athematics and its {A}pplications},
title = {Handbook of {C}ategorical {A}lgebra 1 -- {B}asic {C}ategory {T}heory},
volume = {50},
year = {1994},
}

@book{lam,
author = {Lam, T. Y.},
publisher = {Springer},
title = {Lectures on Modules and Rings},
year = {1999},
}

@book{maclane,
author = {Mac Lane, S.},
publisher = {Springer},
series = {Graduate {T}exts in {M}athematics},
title = {Categories for the {W}orking {M}athematician -- Second Edition},
volume = {5},
year = {1998},
}

@article{abrvic,
author = {Abramsky, S. and Vickers, S.},
journal = {Math. Structures Comput. Sci},
pages = {161-227},
title = {Quantales, observational logic and process semantics},
volume = {3},
year = {1993},
}

@article{borvdb,
author = {Borceux, F. and Van Den Bossche G.},
journal = {Topol. Appl},
pages = {203-223},
title = {An essay on non-commutative topology},
volume = {31},
year = {1989},
}

@article{conmir,
author = {Coniglio, M. E. and Miraglia, F.},
journal = {Studia Logica},
pages = {223-236},
title = {Non-Commutative Topology and Quantales},
volume = {65},
year = {2000},
}

@article{galtsi,
author = {Galatos, N. and Tsinakis, C.},
journal = {Journal of Symbolic Logic},
number = {3},
pages = {780-810},
title = {Equivalence of consequence relations: an order-theoretic and categorical perspective},
volume = {74},
year = {2009},
}

@incollection{krupas,
author = {Kruml, D. and Paseka, J.},
booktitle = {Handbook of Algebra, Vol. 5},
editor = {M. Hazewinkel},
publisher = {Elsevier},
title = {Algebraic and Categorical Aspects of Quantales},
year = {2008},
}

@article{mul,
author = {Mulvey, C. J.},
journal = {Supplemento ai Rendiconti del Circolo Matematico di Palermo, II},
pages = {99-104},
title = {{\&}},
volume = {12},
year = {1986},
}

@incollection{mulnaw,
address = {Kluwer Acad. Publ., Dordrecht},
author = {Mulvey, C. J. and Nawaz, M.},
booktitle = {Theory Decis. Lib. Ser. B Math. Statist. Methods},
pages = {159-217},
publisher = {Vol. 32},
title = {Quantales: Quantale Sets},
year = {1995},
}

@article{pas,
author = {Paseka, J.},
journal = {Cahiers Topologie G\'eom. Diff\'erentielle Cat\'eg},
pages = {19-34},
title = {A note on nuclei of quantale modules},
volume = {XLIII},
year = {2002},
}

@book{rosenthal,
author = {Rosenthal, K. I.},
publisher = {Longman Scientific and Technical},
title = {Quantales and their applications},
year = {1990},
}

@phdthesis{rusthesis,
author = {Russo, C.},
school = {University of Salerno -- Italy},
title = {Quantale Modules, with Applications to Logic and Image Processing},
year = {2007},
}

@article{rusjlc,
author = {Russo, C.},
journal = {Journal of Logic and Computation},
number = {4},
pages = {917-946},
title = {Quantale Modules and their Operators, with Applications},
volume = {20},
year = {2010},
}

@article{rusapal,
author = {Russo, C.},
journal = {Annals of Pure and Applied Logic},
number = {2},
pages = {112-130},
title = {An order-theoretic analysis of interpretations among propositional deductive systems},
volume = {164},
year = {2013},
}

@article{ruscorrapal,
author = {Russo, C.},
journal = {Annals of Pure and Applied Logic},
number = {3},
pages = {392-394},
title = {Corrigendum to ``{A}n order-theoretic analysis of interpretations among propositional deductive systems'' [{A}nn. {P}ure {A}ppl. {L}ogic 164 (2) (2013) 112–130]},
volume = {167},
year = {2016},
}

@article{russajl,
author = {Russo, C.},
journal = {South American Journal of Logic},
number = {2},
pages = {405-424},
title = {Quantales and their modules: projective objects, ideals, and congruences},
volume = {2},
year = {2016},
}

@article{ruslu,
author = {Russo, C.},
journal = {Logica Universalis},
number = {1-2},
pages = {355-380},
title = {Coproduct and Amalgamation of Deductive Systems by Means of Ordered Algebras},
volume = {16},
year = {2022},
}

@article{sol,
author = {Solovyov, S. A.},
journal = {Algebra Universalis},
pages = {35-58},
title = {On the category {$Q$}-Mod},
volume = {58},
year = {2008},
}

@article{yetter,
author = {Yetter, D. N.},
journal = {Journal of Symbolic Logic},
number = {1},
pages = {41-64},
title = {Quantales and (Noncommutative) Linear Logic},
volume = {55},
year = {1990},
}

@article{kat7,
author = {Katsov, Y.},
journal = {Algebra Univ},
pages = {197-214},
title = {Toward homological characterization of semirings: {S}erre's conjecture and Bass's perfectness in a semiring context},
volume = {52},
year = {2004},
}

@article{katnam,
author = {Katsov, Y. and Nam, T. G.},
journal = {Journal of Algebra and its Applications},
number = {3},
pages = {445-473},
title = {Morita equivalence and homological characterization of semirings},
volume = {10},
year = {2011},
}

@article{zhang2022,
title = {Sober topological spaces valued in a quantale},
journal = {Fuzzy Sets and Systems},
volume = {444},
pages = {30-48},
year = {2022},
issn = {0165-0114},
doi = {https://doi.org/10.1016/j.fss.2021.11.007},
url = {https://www.sciencedirect.com/science/article/pii/S0165011421004231},
author = {Zhang, D.-X. and Zhang, G.}
}

\end{document}